\documentclass[a4paper,11pt,reqno]{amsart}
\usepackage[utf8]{inputenc}
\usepackage[dvips]{graphicx}
\usepackage[shortlabels]{enumitem}
\usepackage{amsmath,amssymb,bigints,color,enumitem,graphics,hyperref,latexsym,pgfplots,tikz-cd,vmargin}
\hypersetup{
unicode=false,          
pdftoolbar=true,        
pdfmenubar=true,        
pdffitwindow=false,     
pdfstartview={FitH},    
pdftitle={},    
pdfauthor={},     
pdfsubject={},   
pdfkeywords={}, 
pdfnewwindow=true,      
colorlinks=true,       
linkcolor=black,          
citecolor=black,        
filecolor=magenta,      
urlcolor=cyan           
}

\def\C{{\mathbb{C}}}

\def\d{\,{\mathrm{d}}}

\def\Hh{{\mathcal{H}}}

\def\Ll{{\mathcal{L}}}

\def\N{{\mathbb{N}}}

\def\R{{\mathbb{R}}}
\def\T{{\mathbb{T}}}
\def\Ww{{\mathcal{W}}}

\def\Z{{\mathbb{Z}}}

\newcommand{\norm}[1]{{\left\|{#1}\right\|}}
\newcommand{\abs}[1]{{\left|{#1}\right|}}
\newcommand{\scal}[1]{{\left\langle{#1}\right\rangle}}

\newtheorem{theorem}{Theorem}[section]

\newtheorem{proposition}[theorem]{Proposition}
\newtheorem{corollary}[theorem]{Corollary}

\theoremstyle{definition}
\newtheorem{definition}[theorem]{Definition}
\newtheorem{example}[theorem]{Example}

\newtheorem{remark}[theorem]{Remark}

\newtheorem{conjecture}[theorem]{Conjecture}

\begin{document}

\title{Weaving Riesz Bases}

\let\thefootnote\relax\footnote{2020 {\it Mathematics Subject Classification:} Primary 42C15, 47A15,  94A20.

{\it Keywords:} Frame-tuples, similarity, invariant subspaces, reducing subspaces, shift operators, synthesis operator.

The research of the authors is partially supported by grants: UBACyT 20020170100430BA, PICT 2014-1480 (ANPCyT) and CONICET PIP 11220150100355. The third author was supported by the postdoctoral fellowship 10520200102714CO of CONICET.}
\thanks{OrcID[C. Cabrelli]: 0000-0002-6473-2636}
\thanks{OrcID[U. Molter]: 0000-0002-2928-9480}

\author[C. Cabrelli, U. Molter, F. Negreira]{C. Cabrelli, U. Molter, F. Negreira}
\address{ Departamento de Matemática, Universidad de Buenos Aires, Instituto de Matemśtica ``Luis Santaló'' (IMAS-CONICET-UBA), Buenos Aires, Argentina}
\email{carlos.cabrelli@gmail.com}
\email{umolter@gmail.com}
\email{fnegreira@dm.uba.ar}

\begin{abstract}
This paper explores woven frames in separable Hilbert spaces with an initial focus on the finite-dimensional case. 
We begin by simplifying the problem to bases, for which we obtain a unique characterization.
We establish a condition that is both necessary and sufficient for vector reconstruction, which is applicable to Fourier matrices. 
Furthermore, we show that these characterizations are still valid in the infinite dimensional case, for Riesz bases. 
Finally, we obtain several results for weaving Riesz bases of translations.
\end{abstract}

\dedicatory{To our dear friend Charly, in appreciation of many years of friendship, enriching mathematical conversations and beautiful nights at the opera.}

\maketitle

\section{Introduction}
Let $\Hh$ be a separable Hilbert space and $\{v_i\}_{i\in I}$ a countable subset of vectors of $\Hh$. We say that $\{v_i\}_{i\in I}$ is a frame in $\Hh$ if there exists constants $A,B>0$ such that
\begin{equation*}
A\norm{f}^2\leqslant\sum_{i\in I}\abs{\scal{f,v_i}}^2\leqslant B\norm{f}^2
\end{equation*}
holds for all $f\in\Hh$.

Now for the same index set $I$ consider another frame $\{w_i\}_{i\in I}$. We say that $\{v_i\}_{i\in I}$ and $\{w_i\}_{i\in I}$ are {\it woven} if for every $J \subset I$ the sequence $\{v_i\}_{i\in J} \cup \{w_i\}_{i\in I\setminus J}$ is a frame of $\Hh$.
This notion naturally extends to the case of $n$ frames, where partitions of the index set $I$ into $n$ subsets are considered.

Originally introduced in \cite{weaving}, the motivation for studying woven frames stems from questions in distributed signal processing, with potential applications in areas such as wireless sensor networks and signal preprocessing.  Subsequent research has broadened the concept, applying it to generalized frames such as for example Shauder bases \cite{CFL16}, g-frames \cite{gframes}, fusion frames \cite{fusionf} and K-frames \cite{Kframes} and multi-window Gabor frames \cite{multiw}.

In this note we focus first on woven finite frames (Section \ref{finite}), specifically examining the scenario with two frames.  We limit our investigation to the case in which the frames are Riesz bases, exploiting the fact that in finite dimensions, every frame contains a basis. 
Recall that a Riesz basis in a Hilbert space is the image of an orthonormal basis under a bounded, invertible operator.
Additionally, we restrict our study to woven frames 'up to permutation', where two bases are considered woven if there exists a permutation of one that makes them woven.

Even in this seemingly straightforward case, the complexity of the problem becomes apparent, giving rise to a multitude of challenging aspects and numerous unresolved questions. 

We provide various illustrative examples to shed light on the intricacies involved and comment on some of the inquiries that remain open.

We  establish necessary and sufficient conditions, along with several sufficient conditions, for bases to be woven. We  analyze the case of bases of complex exponentials  and establish connections with the Fourier transform on  finite groups. In particular, we prove  that when $p$ is prime we can always  reconstruct a vector $x \in \C^p$  or its Fourier transform $\widehat{x},$ from any weaving of the components of $x$ and $\widehat{x}.$ We also propose a conjecture regarding the minors of a Fourier matrix.
We extend our findings to the infinite-dimensional case in Section \ref{infinite}, providing new necessary and sufficient conditions for woven Riesz bases in this context. 

Finally, in Section \ref{SIS} we  consider the case of Riesz bases  of translations in shift-invariant spaces (SIS). We study two problems. First, we give necessary and sufficient conditions
in order that two sets of Riesz generators of a SIS can be woven without losing the property of being a Riesz set of generators. For this characterization we
use the fiberization techniques of the range function associated to the SIS \cite{bow}.
In the second problem, we provide sufficient conditions for two Riesz bases of translates to be woven. As a particular example we consider the condition in a Paley-Wiener space.

Throughout this text we will always consider Hilbert spaces over $\C$ as a field but all results -excepting the ones in $\S$\ref{Fourier-m}- are valid over $\R$ as well.

In this article we will use a fact that is observed in \cite{weaving}: The property of being woven is preserved by bounded invertible linear maps. This is, if $\Hh_1$ and $\Hh_2$ are Hilbert spaces and $T\in\Ll(\Hh_1,\Hh_2)$  is invertible then two frames in $\Hh_1$ are woven if and only if their images by $T$ are woven in $\Hh_2$.

\section{Finite case}\label{finite}
As one might suspect, weaving frames are highly susceptible to order. For example if $\{v_i\}_{i\in I}$ is an orthonormal basis then any of its non-trivial permutations are not woven with $\{v_i\}_{i\in I}$ (in fact not even complete). This motivates the following variation of the initial question: when are two given frames $\{v_i\}_{i\in I}$ and $\{w_i\}_{i\in I}$ woven `up to permutations'? Since we have not found this definition elsewhere let us precise what we mean by this.
\begin{definition}
We say that two frames $\{v_i\}_{i\in I}$ and $\{w_i\}_{i\in I}$ indexed by the same set $I$ in the same Hilbert space are {\it  woven up to permutations}, if there exists a permutation $\sigma$ of the index set $I$, such that  $\{v_{\sigma(i)}\}_{i\in I}$ is woven with $\{w_i\}_{i\in I}$.
\end{definition}

In finite dimensional spaces, frames are just generators of the space, so if $\{v_1, \dots, v_m\}$ and $\{w_1, \dots, w_m\}$ are frames in a vector space $V$ of dimension $n$, then there exist two permutations $\sigma_1,\sigma_2$ of $\{1, \dots, m\}$ such that $\{v_{\sigma_1(1)}, \dots,  v_{\sigma_1(n)}\}$ and $\{w_{\sigma_2(1)}, \dots, w_{\sigma_2(n)}\}$ are bases. 

We have the immediate but surprising property.
\begin{proposition} Let
$\{v_1, \dots, v_m\}$ and $\{w_1, \dots, w_m\}$ be frames in a vector space $V$ of dimension $n$, and let $\sigma_1,\sigma_2$ be permutations of $\{1, \dots, m\}$ such that $\{v_{\sigma_1(1)}, \dots,  v_{\sigma_1(n)}\}$ and $\{w_{\sigma_2(1)}, \dots, w_{\sigma_2(n)}\}$ are bases. 

If  $\{v_{\sigma_1(1)}, \dots, v_{\sigma_1(n)}\}$ and $\{w_{\sigma_2(1)}, \dots, w_{\sigma_2(n)}\}$  are woven, then denoting by $\sigma :=\sigma_1^{-1}\sigma_2$ we  have that $\{v_1, \dots,  v_m\}$ and $\{w_{\sigma(1)}, \dots,  w_{\sigma(m)}\}$ are woven. 
\end{proposition}

We do not know however if the converse statement is true, that is: given two finite frames which are woven, do they contain woven bases? 

Also, let us remark at this point, that there exist bases for which no permutation makes them woven:
\begin{example}
In $\C^3$ take $\{e_1,e_2,e_3\}$ to be the canonical basis and $\{e_1+e_2,e_2+e_3,e_1+e_2+e_3\}$  the other basis. There exists no permutation that makes these bases woven. Actually, since being woven is an invariant property under isomorphisms, in an arbitrary $3$-dimensional vector space $V$ {\it any} base $\{v_1,v_2,v_3\}$ is not woven with $\{v_1+v_2,v_2+v_3,v_1+v_2+v_3\}$
\end{example}

This leads us to study when two bases in a finite dimensional space are woven. As we will see, this can be characterized through the change of bases matrix.

\subsection{A necessary and sufficient condition}

\begin{definition}
Given two bases $\{v_1, \dots, v_n\}$ and $\{w_1, \dots, w_n\}$ in a finite dimensional real (or complex) vector space, we say that $A=(a_{i,j})^n_{i,j=1}$ is the change of basis matrix from $\{v_1, \dots, v_n\}$ to  $\{w_1, \dots, w_n\}$ if 
\begin{equation*}
A^t\left(\begin{array}{l}v_1\\\hspace*{0.2em}\vdots\\v_n\end{array}\right)=\left(\begin{array}{l}w_1\\\hspace*{0.2em}\vdots\\w_n\end{array}\right)
\end{equation*}
in the sense that for each $i=1, \dots, n$ we have
\begin{equation*}
w_i=a_{1,i}v_1 + \dots + a_{n,i}v_n.
\end{equation*}
\end{definition}

Particularly, we will be looking at some sub-matrices  of this matrix, which we call {\it central}.

\begin{definition}
The square {\it central} sub-matrices  of a real or complex matrix $A=(a_{i,j})^n_{i,j=1}$ are those that can be written as $(a_{j,k})_{j,k\in J}$ where $J$ is any subset of $\{1, \dots, n\}$.
\end{definition}

We have the following theorem.
\begin{theorem}\label{elteo}
Let $V$ be a vector space of finite dimension, Let $\{v_1, \dots, v_n\}$, $\{w_1, \dots, w_n\}$ two bases of $V$ and denote by $A$ the change of basis matrix. Then $\{v_1, \dots, v_n\}$ and $\{w_1, \dots, w_n\}$ are woven if and only if all central sub-matrices  of $A$ are invertible.
\end{theorem}
\begin{proof}
Define $T:V\to\C^n$ as the isomorphism such that $Tv_j=e_j$ for all $j=1,\dots,n$ where
$\{e_1,\dots,e_n\}$ is the canonical base of $\C^n$.

As $A=(a_{i,j})^n_{i,j=1, \dots, n}$ is the change of basis matrix from $\{v_1, \dots, v_n\}$ to $\{w_1, \dots, w_n\}$ then for each $i=1, \dots, n$ we have
\begin{equation*}
w_i=a_{1,i}v_1+ \dots + a_{n,i}v_n.
\end{equation*}
Thus $Tw_i$ is the $i^{th}$-column of $A$, $a_i=(a_{1,i}, \dots, a_{n,i})$. Then, as being woven is invariant under isomorphisms, $\{v_1, \dots, v_n\}$ and $\{w_1, \dots, w_n\}$ are woven in $V$ if and only if $\{e_1, \dots, e_n\}$ and $\{a_1, \dots, a_n\}$ are woven in $\C^n$. 

But $\{e_1, \dots, e_n\}$ and $\{a_1, \dots, a_n\}$ are woven if for any subset $J\subset\{1, \dots, n\}$ the set $\{e_i\}_{i\in J}\cup\{a_i\}_{i\in I\setminus J}$ is a basis. In other words, the matrix $A(J)$ in which the columns of $A$ indexed in $J$ are changed by the corresponding canonical vectors, must be invertible. 

Calculating the determinant using the expansion by precisely those rows indexed by elements in $J$ we see that $\det(A(J))$ is the determinant of the central sub-matrix of $A$, $(a_{i,j})_{i,j\in J}$. And all such determinants can be obtained by choosing an appropriate subset of $\{1, \dots,  n\}$.
\end{proof}

\subsection{The class $\Ww$}

\begin{definition}
We define the class $\Ww$ as the set of square matrices   with entries in $\C$ such that all its central sub-matrices   are invertible. The subclass $\Ww_n$ are the elements of $\Ww$ of size $n\times n$.
\end{definition}

\begin{corollary}\label{final}
Assume that $A\in\Ww_n$ and let $J\subset\{1, \dots, n\}$. Then any  $x\in\C^n$ can be recovered from the samples $Y_J(x)=\{(Ax)(j)\}_{j\in J}\cup\{x(j)\}_{j\in I\setminus J}$ (i.e. the map $x\mapsto Y_J(x)$ is injective.) Moreover, the converse result is also true, this is, if for a given $A\in\C^{n\times n}$ the map $x\mapsto Y_J(x)$ is injective for all $J\subset\{1, \dots, n\}$ then $A\in\Ww_n$.
\end{corollary}
\begin{proof}
Given $J\subset\{1, \dots, n\}$ let $A(J)$ be the matrix whose $j^{\rm th}$-row is equal to $Ae_j$ when $j\in J$ and $e_j$ when $j\in I\setminus J$. Thus, $Y_J(x)=A(J)x$.

Assume now $A=(a_{i,j})_{i,j=1}^n\in\Ww_n$. Then, as $\det A(J)=\det(a_{i,j})_{i,j\in J}$, $A(J)$ is invertible and we can therefore write $x=A(J)^{-1}(A(J)(x))$.

Conversely, if $x\mapsto Y_J(x)$ is injective then multiplication by $A(J)$ is also an injective map. Furthermore, as $A(J)$ is a square matrix, then it must be invertible. Thus $\det(a_{i,j})_{i,j\in J}=\det A(J)\neq0$.
\end{proof}

\begin{remark}
More generally, if $A$ and $B$ are invertible matrices  in $\C^n$, then the rows of $A$  are woven with the rows of $B$ if and only if for any $J\subset\{1, \dots, n\}$ any vector $x\in\C^n$ can be uniquely recovered from  $\{(Ax)(j): j\in J\}\cup\{(Bx)(j): j\in I\setminus J\}$. This can also be seen by applying the generalized Cramer's rule, see e.g. \cite{Cramers}.
\end{remark}

Just like with other properties, it is natural to wonder how often we come across matrices in $\Ww$? Or, more rigorously, is the set $\Ww$ generic in some sense (topologically, probabilistically)? In fact, since the determinant is a polynomial function on the entries of a matrix, matrices in  $\Ww_n$ are in the complement of a manifold of dimension strictly lower than $n\times n$. This is, a Zariski open set, see e.g. \cite{zariski}.

\begin{proposition}
Given a natural number $n$, the set of matrices  in $\Ww_n$ is a Zariski (dense) open set within the matrices  in $\C^{n\times n}$.
\end{proposition}
\begin{proof}
For each $J\subset\{1, \dots, n\}$ consider the function $P_J:\C^{n\times n}\to\C$ given by $P_JA=\det(a_{j,k})_{j,k\in J}$ where  $A=(a_{j,k})_{j,k=1}^n$. Clearly, $P_J$ is a polynomial on $n\times n$ variables and, by definition, $\Ww_n^c=\cup_JP_J^{-1}(0)$. This means that $\Ww_n^c$ is a Zariski closed set.
\end{proof}

Viewed this way, the subclass $\Ww_n$ is a huge subset of $\C^{n\times n}$. Finding a structure for this set, however, can be tricky. It is not a vector space with the usual sum since clearly $0\notin\Ww_n$, nor is it closed by multiplication e.g. $A=\left(\begin{array}{rr}1&-1\\1&1\end{array}\right)\in\Ww_2$ but $A^2=\left(\begin{array}{rr}0&-2\\2&0\end{array}\right)\notin\Ww_2$. It has, nonetheless, some symmetries and other properties that we list below.

\begin{proposition}\label{propertiesf}
Let $A$ be a complex square matrix.
\begin{enumerate}
\item $A\in\Ww$ if and only if $A$ is invertible and $A^{-1}\in\Ww$. 
\item $A\in\Ww\Leftrightarrow A^t \in\Ww\Leftrightarrow A^*\in\Ww$. 
\item If $A \in\Ww_n$ and $B\in\C^{n\times n}$ is either a diagonal invertible or a permutation matrix then $B^*AB\in\Ww_n$.

\item If $A$ is symmetric and either positive definite or negative definite, then $A\in\Ww$.

\item If $A$ is lower or upper triangular and invertible then $A\in\Ww$.

\item If $A \in \Ww$ and $B$ is a central sub-matrix of $A$, then $B$ is in $\Ww$.
\end{enumerate}
\end{proposition}
\begin{proof}
$(1).$ From Theorem~\ref{elteo}, we know that $A\in\Ww$ if and only if $\{e_1, \dots,  e_n\}$ and $\{Ae_1, \dots, Ae_n\}$ are woven bases. And using $A^{-1}$ as an isomorphism we have that $\{e_1, \dots,  e_n\}$ and $\{Ae_1, \dots, Ae_n\}$ are woven if and only if 
\begin{equation*}
\{A^{-1}e_1, \dots, A^{-1}e_n\}\text{ and }\{A^{-1}Ae_1, \dots, A^{-1}Ae_n\}=\{e_1, \dots,  e_n\}
\end{equation*}
are woven bases, which then again by Theorem~\ref{elteo} is equivalent to $A^{-1} \in \Ww$.

\bigskip

\noindent$(2).$ This follows from the fact that central sub-matrices  of $A^t$ and $A^*$ are just the transpose and adjoint respectively of those in $A$.

\bigskip

\noindent$(3).$ A central sub-matrix of $B^*AB$ for $B$ a diagonal invertible matrix, is a central sub-matrix of $A$ with rows and columns multiplied by non-zero complex numbers. The result then follows from the linearity of the determinant.

As for permutation matrices, it is enough to show the invariance under a single permutation since the set of matrices  that fix property under conjugation $\Ww$ is a group. Consider then $R_{i,j}$ the matrix that switches the $i^{\rm th}$-row with the $j^{\rm th}$-row and define $\widetilde{A}:=(R_{i,j})^*AR_{i,j}$. Then $\widetilde{A}=R_{i,j}AR_{i,j}$ is the matrix constructed by first switching rows $i$ and $j$, and then switching columns  $i$ and $j$. Take now central sub-matrix $\widetilde{A}_S$ of $\widetilde{A}$, $S\subset\{1, \dots, n\}$. If $a_{i,i}$ and $a_{j,j}$ are not entries in $\widetilde{A}_S$ then neither the rows nor the columns $i, j$ are in $\widetilde{A}_S$ and so $\widetilde{A}_S=A_S$ the corresponding central sub-matrix in $A$. If both $a_{i,i}$ and $a_{j,j}$ are entries in $\widetilde{A}_S$ then $\widetilde{A}_S$ is a matrix obtained by row and column permutations of $A_S$ and so $\det(\widetilde{A}_S)=\det(A_S)$. If just one of either $a_{i,i}$, or $a_{j,j}$ is an entry in $\widetilde{A}_S$, say $a_{i,i}$, then $\widetilde{A}_S$ is a matrix obtained by row and column permutations of $A_{S'}$ the square sub-matrix of $A$ with $S'=S\setminus\{i\}\cup\{j\}$.

\bigskip

\noindent$(4).$ If $A \notin\Ww_n$ then there exist a central sub-matrix $A_J$ which is not invertible. This means that there exist a non-trivial vector $y\in\ker A_J$. Define now $x\in\C^n$ with $x_i=0$ for all $i\notin J$ and $x_i=y_i$ for $i\in J$. Thus $x^*Ax=0$ and $x\neq 0.$ In particular $A$ cannot be symmetric definite.

\bigskip

\noindent$(5).$ This follows from the fact that central sub-matrices  of a triangular matrix are triangular, and they are invertible if and only if all its diagonal entries are non-zeros. 

\bigskip

\noindent$(6).$ This is straightforward.
\end{proof}

We conjecture that the matrices $B\in\C^{n\times n}$ listed in $(3)$ are the only ones that leave the subclass $\Ww_n$ invariant under conjugation. Arbitrary unitary matrices, for example, can be outside this set. Indeed, let $A=R_{1,2}I\in\C^{n\times n}$ be the matrix consisting of switching the first two rows of the identity. Then $A\notin\Ww_n$ but since $A$ is unitary and invertible there exists another unitary matrix $U$ such that $U^*AU$ is diagonal and invertible and thus belongs to $\Ww_n$.

In general, permutation by columns or rows of matrix in $\Ww$ needs not to remain in $\Ww$. Actually for this to happen all the minors of the matrix must be non-zero.

\begin{proposition}\label{col-perm}
Let $A$ be a square complex matrix. Then all column permutations of $A$ belong to $\Ww$ if and only if all the minors of $A$ are non-zero.
\end{proposition}
\begin{proof}
Assume first that $A=(a_{i,j})_{i,j=1}^n$ and all its permutations by columns are in $\Ww$, and take $M=(a_{i_s,j_s})_{s=1}^k$ a square sub-matrix of $A$. Let $P_{i_s,j_s}$ be the permutation that switches the $i_s$-column with the $j_s$-column, and define $P:=P_{i_1,j_1}\dots P_{i_k,j_k}$. Then $M$ is a central sub-matrix of $PA$. Indeed, the entry in $(i_1,i_1)$ of $PA$ is $a_{i_1,j_1}$ and the entry in $(i_k,i_k)$ is $a_{i_k,j_k}$. Thus, as $PA\in\Ww$, $\det M\neq0$.

Conversely, suppose now that all minors of $A$ are non-zero and take $P$ a permutation column matrix. Then, central sub-matrices of $PA$ are just permutations of square sub-matrices of $A$. Therefore $PA\in\Ww$.
\end{proof}

\subsection{Fourier matrices}\label{Fourier-m}
Fourier matrices are good candidates to belong to the class $\Ww$. Recall that, given a natural number $n$, the Fourier matrix of order $n$ is defined by $F_n=\{e^{2\pi ijk/n}\}_{k,j=0}^{n-1}$. It is known for example that {\bf all} square sub-matrices of $F_n$ are invertible if and only if $n$ is a prime number (see e.g. \cite{tao} and the references therein). This means that $F_n$ satisfies Proposition~\ref{col-perm} if and only if $n$ is prime. In particular this means that for $n$ prime, $F_n\in\Ww_n$. In this subsection we try to identify if there are other natural numbers $n$ for which $F_n$ is in $\Ww_n$.

To begin with, to see whether a Fourier matrix is in $\Ww$ it suffices to check the invertibility of a subset of its central sub-matrices.

\begin{proposition}\label{central-1}
The Fourier matrix  $F_n\in\Ww_n$ if and only if all its central sub-matrices $(F_n)_J$ with $1\in J$ are invertible.
\end{proposition}
\begin{proof}
Let $\zeta=e^{2\pi i/n}$. Then $F_n=(\zeta^{j \ell})_{\ell,j=0}^{n-1}$. For $0\leqslant j_0<\dots<j_k<n$, we consider the following central $k+1 \times k+1$ sub-matrix of $F_n$:
\begin{equation*}
\left(\begin{array}{lll}\zeta^{j_0^2}&\dots&\zeta^{j_0j_k}
\\
\vdots&\ddots&\vdots
\\
\zeta^{j_{k}j_0}&\dots&\zeta^{j_{k}^2}\end{array}\right)
\end{equation*}
Multiplying each row by the inverse of its first element,  by linearity of the determinant with respect to rows, we obtain
\begin{equation*}
\det\left(\begin{array}{lll}\zeta^{j_0^2}&\dots&\zeta^{j_0j_k}\\\vdots&\ddots&\vdots\\\zeta^{j_kj_0}&\dots&\zeta^{j_k^2}\end{array}\right)=\prod_{s=0}^k\zeta^{-j_0 j_s}\det\left(\begin{array}{lll}1&\dots&\zeta^{j_0(j_k-j_0)}\\\vdots&\ddots&\vdots\\1&\dots&\zeta^{j_k(j_k-j_0)}\end{array}\right).
\end{equation*}
Now, multiplying the resulting $k+1$-columns by the inverse of the top element we have
\begin{equation*}
\det\left(\begin{array}{lll}1&\dots&\zeta^{j_0(j_k-j_0)}\\\vdots&\ddots&\vdots\\1&\dots&\zeta^{j_k(j_k-j_0)}\end{array}\right)=\prod_{s=0}^k\zeta^{-j_0(j_s-j_0)}\det\left(\begin{array}{lll}1&\dots&1\\\vdots&\ddots&\vdots\\1&\dots&\zeta^{(j_k-j_0)^2}\end{array}\right).
\end{equation*}
So that if $l_s :=j_s-j_0, s= 0, \dots, k$,  the entries of this last matrix are $(\zeta^{l_sl_r})_{r,s=1}^k$ which is a central sub-matrix of $F_n$ with its first entry $1$ since $0=l_0<\dots<l_k<n$.
\end{proof}

\begin{corollary}
Let $n\geqslant2$ be a natural number. Then all central $2\times 2$ sub-matrices  of $F_n$ are invertible if and only if $n$ is square free.
\end{corollary}
\begin{proof}
By the argument in the proof of 
Proposition \ref{central-1} we know that any central sub-matrix of $F_n$ is invertible if and only if there exists an invertible central sub-matrix of the same order whose set of indexes contain 1.
As a consequence of that and the fact that the elements in the diagonal of $F_n$ are of the form $\zeta^{j^2}$ for $j=0, \dots, n-1$, we only need to look at the case when the considered central sub-matrix is of the form,
\begin{equation*}
\left(\begin{array}{ll}1&1\\1&\zeta^{j^2}\end{array}\right) \quad \text{ for some $0<j<n$.}
\end{equation*} 
The determinant of this sub-matrix is $\zeta^{j^2}-1$ and thus it is invertible if and only if $\zeta^{j^2}\neq1$. As $\zeta=e^{2\pi i/n}$, then $\zeta^{j^2}=1$ if and only if $n$ divides $j^2$ i.e. $j^2 = kn$.

Suppose first that $n$ is square free and let $n=p_1\dots p_m$ be its decomposition into primes with $p_r\neq p_s$ for $r\neq s$. If $n$ divides $j^2$ then $p_r$ divides $j^2$ for $1\leqslant r\leqslant m$. But then $p_r$ divides $j$ for $1\leqslant r\leqslant m$, so that $n$ divides $j$ which is impossible since $0<j<n$.

If $n$ is not square free then we can decompose it as $n=p^2q$ with $p$ prime. Take $j=pq$. Then $0<j<n$ and $n$ divides $j^2$, so that there exists a $2\times2$ central sub-matrix of $F_n$ which is not invertible.
\end{proof}

\begin{corollary}\label{Fourier}
If $n\geqslant2$ is a natural number such that $F_n\in\Ww$ then $n$ is square free.
\end{corollary}
That is, for $n$ to be square-free is a necessary condition for $F_n\in\Ww$.

We conjecture that the converse result of Corollary \ref{Fourier} is also true. i.e. 
\begin{conjecture}
If $n$ is square free then $F_n\in\Ww$. That is, if $n$ is square free then every central submatrix of $F_n$ is invertible.
\end{conjecture}
This would make for a larger subset than only prime numbers to $F_n$ belong to $\Ww$. 

\begin{remark}
Let $\widehat{x}$ denote the discrete Fourier transform of $x\in\C^n$, that is $\widehat{x}=F_nx$. Applying Corollary \ref{final} for the case of Fourier matrices  which are in $\Ww$ reads as follows: if $n$ is such that $F_n\in\Ww_n$ then for any $J\subset\{1, \dots, n\}$ and any $x\in\C^n$ we can reconstruct $x$ from the measurements $\{\widehat{x}(j)\}_{j\in J}\cup\{x(j)\}_{j\in I\setminus J}$.
\end{remark}

\section{Infinite case, Riesz basis}\label{infinite}

We move now to the case of Riesz bases of infinite-dimensional separable Hilbert spaces. Let us first recall the definition of a Riesz basis: in a Hilbert space $\Hh$ we say that a countable subset of vectors $\{v_i\}_{i\in I}$ is a {\it Riesz sequence} if there exists constants $A,B>0$ such that
\begin{equation*}
A\norm{c}_2\leqslant\norm{\sum_{i\in I}c_iv_i}\leqslant B\norm{c}_2
\end{equation*}
holds for all finite sequences of coefficients $c=(c_i)_{i\in I}$. We say that $\{v_i\}_{i\in I}$ is a {\it Riesz basis} if in addition it is a frame.

 A result in \cite[Theorem 5.3]{weaving} shows that if two Riesz bases $\{v_i\}_{i\in I}$ and $\{w_i\}_{i\in I}$ are woven frames then every `weaving choice' $\{v_i\}_{i\in J}\cup\{w_i\}_{i\in J ^c}$ is actually a Riesz basis.

We choose Riesz bases because they are the generalization that best fits  the results of the previous section. Indeed, if $\{v_i\}_{i\in I}$ and $\{w_i\}_{i\in I}$ are two Riesz basis in a Hilbert space $\Hh$ then there is an isomorphism $T:\Hh\to\Hh$ that sends $\{v_i\}_{i\in I}$ into $\{w_i\}_{i\in I}$. Going down to $\ell^2(I)$ via the isomorphism $D_v:\Hh\to\ell^2(I)$ that sends $\{v_i\}_{i\in I}$ to the canonical basis $\{e_i\}_{i\in I}$ then gives an isomorphism $A^t:\ell^2(I)\to\ell^2(I)$. In other words, we have the following commuting diagram
\begin{equation*}
\begin{tikzcd}
\Hh\arrow{r}{T}\arrow[swap]{d}{D_v}&\Hh\arrow{d}{D_v}
\\
\ell^2(I)\arrow{r}{A^t}&\ell^2(I)
\end{tikzcd}
\end{equation*}
Finally, writing $a_{i,j}=\scal{e_i,A^te_j}$ allows to define the change of basis matrix in this context.

\begin{definition}
Given $\{v_i\}_{i\in I}$ and $\{w_i\}_{i\in I}$ as before we say that $A=(a_{i,j})_{i,j\in I}$ is the change of basis matrix from $\{v_i\}_{i\in I}$ to $\{w_i\}_{i\in I}$ if
\begin{equation*}
A^t\{v_i\}_{i\in I}=\{w_i\}_{i\in I}
\end{equation*}
in the sense that for each $i\in I$ we have
\begin{equation*}
w_i=\sum_{j\in I}a_{j,i}v_j.
\end{equation*}
\end{definition}

As before, the characterization of woven bases will come from the central sub-matrices of the change of basis matrix.

\begin{definition}
As before the square {\it central} sub-matrices  of the complex matrix $A=(a_{i,j})_{i,j\in I}$ are those sub-matrices that can be written as $A_J := (a_{i,j})_{i,j\in J},$ where $J$ is any subset of $I$.
\end{definition}

\begin{theorem}\label{teo20}
Let $\Hh$ be a Hilbert space. Let $\{v_i\}_{i\in I}$, $\{w_i\}_{i\in I}$ be two Riesz bases of $\Hh$ and denote by $A$ the change of basis matrix. Then $\{v_i\}_{i\in I}$ and $\{w_i\}_{i\in I}$ are woven if and only if all central sub-matrices of $A$ define uniformly bounded invertible operators, i.e. there exist constants $C_1$ and $C_2$ such that \;$C_1\norm{\alpha}_{\ell^2(J)}^2\leqslant\norm{A_J \alpha}_{\ell^2(J)}^2\leqslant C_2\norm{\alpha}_{\ell^2(J)}^2$, for any $J\subset I$ and $\alpha\in\ell^2(J)$.
\end{theorem}
\begin{proof}
First, note that the isomorphism $D_v:\Hh\to\ell^2(I)$ that sends $\{v_i\}_{i\in I}$ into the canonical basis $\{e_i\}_{i\in I}$, also sends $\{w_i\}_{i\in I}$ into the columns of $A$. Indeed, 
as
\begin{equation*}
w_i=\sum_{j\in I}a_{j,i}v_j \quad \text{ for all $i\in I$},
\end{equation*}
and then 
\begin{equation*}
D_vw_i=\sum_{j\in I}a_{j,i}Dv_j=\sum_{j\in I}a_{j,i}e_j=:a_i ,
\end{equation*}
where $a_i$ denotes the $i^{\rm th}$-column of $A$. Thus, using $D_v$ as an isomorphism, we have that $\{v_i\}_{i\in I}$ and $\{w_i\}_{i\in I}$ are woven if and only if $\{e_i\}_{i\in I}$ and $\{a_i\}_{i\in I}$ are woven. We can then work in $\ell^2(I)$ in the sense that to prove the theorem it is enough to show that $\{e_i\}_{i\in I}$ and $\{a_i\}_{i\in I}$ are woven if and only if the central sub-matrices of $A$ are uniformly bounded invertible operators.

Given a subset $J\subset I$ we will denote by $E_J:\ell^2(J)\to\ell^2(I)$ the immersion given by
\begin{equation*}
E_J(\alpha)_i=\left\{\begin{array}{ll}\alpha_i&\text{if }i\in J\\0&\text{if }i\notin J.\end{array}\right.
\end{equation*}
Then the adjoint operator, $E_J^*:\ell^2(I)\to\ell^2(J)$ is the restriction to the indexes of $J$. Further, $P_J:=E_JE_J^*:\ell^2(I)\to\ell^2(I)$ is the projection onto the image of $E_J$.
Finally, note that if $A_J=(a_{i,j})_{i,j\in J}$ is a central sub-matrix and $\alpha\in\ell^2(J)$ then 
\begin{equation*}
A_J (\alpha)=E_J^*AE_J (\alpha)
\end{equation*}
whenever $A_J(\alpha)$ is defined.

Assume first that $\{e_i\}_{i\in I}$ and $\{a_i\}_{i\in I}$ are woven and fix $J\subset I$. As $A$ defines a bounded operator in $\ell^2(I)$ we have that for any $\alpha\in\ell^2(J)$
\begin{align*}
\norm{A_J(\alpha)}_{\ell^2(J)}=\norm{E_J^*AE_J(\alpha)}_{\ell^2(J)}\leqslant\norm{AE_J(\alpha)}_{\ell^2(I)}&\leqslant\norm{A}_{\ell^2(I)\to\ell^2(I)}\norm{E_J(\alpha)}_{\ell^2(I)}
\\
&\leqslant\norm{A}_{\ell^2(I)\to\ell^2(I)}\norm{\alpha}_{\ell^2(J)}.
\end{align*}
Therefore, $A_J:\ell^2(J)\to\ell^2(J)$ is a bounded operator and its norm is controlled by an independent constant. Thus, we are left to show that $A_J$ is injective and uniformly bounded from below.

Let $\alpha\in\ell^2(J)$ and note that
\begin{equation*}
\norm{A_J(\alpha)}_{\ell^2(J)}=\norm{E_J^*AE_J(\alpha)}_{\ell^2(J)}=\norm{E_JE_J^*AE_J(\alpha)}_{\ell^2(I)}=\norm{\sum_{i\in J}\alpha_iE_J E_J^*(a_i)}_{\ell^2(I)},
\end{equation*}
where the second inequality holds because $E_J:\ell^2(J)\to\ell^2(I)$ is an immersion. Next notice that
\begin{align*}
\sum_{i\in J}\alpha_iE_JE_J^*(a_i) &=
\sum_{i\in J}\alpha_iP_J(a_i) = \sum_{i\in J}\alpha_i(a_i-P_{I\setminus J}(a_i))=\sum_{i\in J}\alpha_ia_i-\sum_{i\in J}\alpha_iP_{I\setminus J}(a_i)
\\
& =\sum_{i\in J}\alpha_ia_i-P_{I\setminus J}\left(\sum_{i\in J}\alpha_ia_i\right)=\sum_{i\in J}\alpha_ia_i-\sum_{j\in I\setminus J}\left(\sum_{i\in J}\alpha_ia_{j,i}\right)e_j.
\end{align*}
Defining $\gamma\in\ell^2(I)$ as 
\begin{equation*}
\gamma_i:=\left\{\begin{array}{ll}\alpha_i & i\in J
\\
-\sum_{i\in J}\alpha_ia_{j,i} & i\in I\setminus J,\end{array}\right.
\end{equation*}
the previous computation yields
\begin{equation*}
\norm{A_J(\alpha)}_{\ell^2(J)}=\norm{\sum_{i\in J}\gamma_ia_i+\sum_{i\in I\setminus J}\gamma_ie_i}_{\ell^2(I)}.
\end{equation*}

Now, as $\{e_i\}_{i\in I}$ and $\{a_i\}_{i\in I}$ are woven then there exists an independent constant $C>0$ such that
\begin{equation*}
\norm{\sum_{i\in J}\gamma_ia_i+\sum_{i\in I\setminus J}\gamma_ie_i}_{\ell^2(I)}\geqslant C\norm{\gamma}_{\ell^2(I)}.
\end{equation*}
On the other hand
\begin{equation*}
\norm{\gamma}^2_{\ell^2(I)}\geqslant\sum_{i\in J}\abs{\gamma_i}^2=\sum_{i\in J}\abs{\alpha_i}^2=\norm{\alpha}^2_{\ell^2(J)},
\end{equation*}
and thus in conclusion,
\begin{equation*}
\norm{A_J(\alpha)}_{\ell^2(J)}\geqslant C\norm{\alpha}^2_{\ell^2(J)},
\end{equation*}
which is what we wanted to prove.

For the other implication, assume now that all central sub-matrices of $A$ define uniformly bounded invertible operators. We want to show that $\{e_i\}_{i\in I}$ and $\{a_i\}_{i\in I}$ are woven. From \cite[Theorem 5.2]{weaving} we know that it is enough to show that  for any $J\subset I$ the set $\{a_i\}_{i\in J}\cup\{e_i\}_{i\in I\setminus J}$ is a Riesz sequence with uniform bounds independent of $J$.

Since $\{e_i\}_{i\in I}$ and $\{a_i\}_{i\in I}$ are Riesz bases in $\ell^2(I)$ there exist some constants $C_1,C_2>0$ such that
\begin{align*}
\norm{\sum_{i\in J}\alpha_ia_i+\sum_{i\in I\setminus J}\alpha_ie_i}_{\ell^2(I)}\leqslant\norm{\sum_{i\in J}\alpha_ia_i}_{\ell^2(I)}+\norm{\sum_{i\in I\setminus J}\alpha_ie_i}_{\ell^2(I)}&\leqslant C_1\norm{\alpha}_{\ell^2(I)}+C_2\norm{\alpha}_{\ell^2(I)}
\\
&\leqslant(C_1+C_2)\norm{\alpha}_{\ell^2(I)}
\end{align*}
holds for all $\alpha\in\ell^2(I)$ and all $J\subset I$. Therefore, we only need to show the existence of a lower bound $C_3>0$ such that
\begin{equation*}
C_3\norm{\alpha}_{\ell^2(I)}\leqslant\norm{\sum_{i\in J}\alpha_ia_i+\sum_{i\in I\setminus J}\alpha_ie_i}_{\ell^2(I)}
\end{equation*}
holds for all $\alpha\in\ell^2(I)$ and all $J\subset I$. 

We proceed by contradiction. Suppose that for every $\varepsilon >0$ there exists $\alpha\in\ell^2(I)$ and $J\subset I$ with $\norm{\alpha}_{\ell^2(I)}=1$ and such that 
\begin{equation*}
\norm{\sum_{i\in J}\alpha_ia_i+\sum_{i\in I\setminus J}\alpha_ie_i}_{\ell^2(I)}\leqslant\varepsilon.
\end{equation*}
Let $\rho:=\sum_{i\in J}\alpha_ia_i+\sum_{i\in I\setminus J}\alpha_ie_i$. Using the projections defined before, we will now estimate $\norm{P_J(\rho)}$ and $\norm{P_{I\setminus J}(\rho)}$ to yield the desired contradiction. First note that projecting gives 
\begin{equation*}
\norm{P_J(\rho)}_{\ell^2(I)}\leqslant\norm{\rho}_{\ell^2(I)}\leqslant\varepsilon\quad\text{and}\quad\norm{P_{I\setminus J}(\rho)}_{\ell^2(I)}\leqslant\norm{\rho}_{\ell^2(I)}\leqslant\varepsilon.
\end{equation*}
Now, for each $i\in J$ let $b_i\in\ell^2(J)$ be defined by $b_i:=(E_J^*)a_i$. Also let $\beta\in\ell^2(J)$ given by $\beta:=(E_J^*)\alpha$. Thus
\begin{align*}
\norm{P_J(\rho)}_{\ell^2(I)}& =\norm{P_J\left(\sum_{i\in J}\alpha_ia_i+\sum_{i\in I\setminus J}\alpha_ie_i\right)}_{\ell^2(I)}=\norm{\sum_{i\in J}\alpha_iP_J(a_i)}_{\ell^2(I)}=\norm{\sum_{i\in J}\alpha_iE_JE_J^*(a_i)}_{\ell^2(I)}
\\
& =\norm{\sum_{i\in J}\alpha_iE_J(b_i)}_{\ell^2(I)}=\norm{\sum_{i\in J}\alpha_ib_i}_{\ell^2(J)}=\norm{\sum_{i\in J}E_J^*(\alpha_ib_i)}_{\ell^2(J)}=\norm{A_J(\beta)}_{\ell^2(J)}.
\end{align*}
Since $A_J$ is uniformly bounded from below there exists a constant $C>0$ independent of $J$ and such that $C\norm{\beta}_{\ell^2(J)}\leqslant\norm{A_J(\beta)}_{\ell^2(J)}$. Altogether this yields
\begin{equation}\label{jota}
C\norm{\beta}_{\ell^2(J)}\leqslant\norm{A_J(\beta)}_{\ell^2(J)}=\norm{P_J(\rho)}_{\ell^2(I)}\leqslant\norm{\rho}_{\ell^2(I)}\leqslant\varepsilon.
\end{equation}

Next, we look at the projection onto $I\setminus J$:
\begin{align}
\norm{P_{I\setminus J}(\rho)}_{\ell^2(I)} & = \norm{P_{I\setminus J}\left(\sum_{i\in J}\alpha_ia_i\right)+\sum_{i\in I\setminus J}\alpha_ie_i}_{\ell^2(I)}\nonumber
\\
&\geqslant\norm{\sum_{i\in I\setminus J}\alpha_ie_i}_{\ell^2(I)}-\norm{P_{I\setminus J}\left(\sum_{i\in J}\alpha_ia_i\right)}_{\ell^2(I)}. \label{pjc}
\end{align}

For the first term, noting that $\norm{P_J(\alpha)}^2_{\ell^2(I)}=\norm{E_J^*(\alpha)}_{\ell^2(J)}^2=\norm{\beta}_{\ell^2(J)}^2$, we have
\begin{align*}
\norm{\sum_{i\in I\setminus J}\alpha_ie_i}_{\ell^2(I)} & =\norm{P_{I\setminus J}(\alpha)}_{\ell^2(I)}=\norm{\alpha-(P_J)\alpha}_{\ell^2(I)}
\\
& =\sqrt{\norm{\alpha}_{\ell^2(I)}^2-\norm{P_J(\alpha)}_{\ell^2(I)}^2}=\sqrt{1-\norm{\beta}_{\ell^2(J)}^2}\\
&\geqslant 1-\norm{\beta}_{\ell^2(J)},
\end{align*}
since $\norm{\beta}_{\ell^2(J)}\leqslant\norm{\alpha}_{\ell^2(I)}=1$. 

For the second term we have
\begin{equation*}
\norm{P_{I\setminus J}\left(\sum_{i\in J}\alpha_ia_i\right)}_{\ell^2(I)}\hspace*{-0.3em}\leqslant\norm{\sum_{i\in J}\alpha_ia_i}_{\ell^2(I)}\hspace*{-0.3em}=\norm{\sum_{i\in I}P_J(\alpha)_ia_i}_{\ell^2(I)}\hspace*{-0.3em}\leqslant C_1\norm{P_J(\alpha)}_{\ell^2(I)}=C_1\norm{\beta}_{\ell^2(J)},
\end{equation*}
where, as before, $C_1$ denotes the upper bound for the Riesz basis $\{a_i\}_{i\in I}$. Finally, plugging in the estimates for both terms in \eqref{pjc} we have
\begin{equation} \label{jotap}
1-(1+C_1)\norm{\beta}_{\ell^2(J)}\leqslant\norm{P_{I\setminus J}(\rho)}_{\ell^2(I)}\leqslant\norm{\rho}_{\ell^2(I)}\leqslant\varepsilon.
\end{equation}

Combining now \eqref{jota} and \eqref{jotap} we have 
\begin{equation*}
1\leqslant\left(1+\frac{1+C_1}{C}\right)\varepsilon,
\end{equation*}
which, letting $\varepsilon \to0$, gives a contradiction. Therefore $\{a_i\}_{i\in J}\cup\{e_i\}_{i\in I\setminus J}$ must be a Riesz sequence with bounds independent of $J$. This completes the proof.
\end{proof}

\begin{definition}
Similar to the finite dimensional case we define the class $\Ww$ as the set of matrices in $\ell^2(I)$ such that all its central sub-matrices  are invertible.
\end{definition}

Therefore, in view of Theorem \ref{teo20}, we have that two Riesz bases $\{v_i\}_{i\in I}$, $\{w_i\}_{i\in I}$ in a Hilbert space $\Hh$ are woven if and only if its change of basis matrix $A$ belongs to the class $\Ww$.

This result is stronger than the one given in \cite[Proposition 7.1]{weaving} which says that if $\{v_n\}_{n\in\N}$ is a Riesz basis in $\ell^2(I)$ and its Grammian (change of basis with the canonical basis) can be written as $D+R$ where $D$ is diagonal and $2\norm{R}\leqslant\sup_n\abs{d_{n,n}}$ then $\{v_n\}_{n\in\N}$ is woven with $\{e_i\}_{i\in I}$. 

Indeed, one can construct an example which satisfies the hypothesis of Theorem \ref{teo20} but not the ones of \cite[Proposition 7.1]{weaving}: consider the set $\{v_n\}_{n\in\N}$ given by $v_1=e_1+2e_2$, $v_2=2e_1+e_2$ and $v_n=e_n$ for $n\geqslant3$. Then clearly $\{v_n\}_{n\in\N}$ is a Riesz basis and the change of basis matrix from the canonical basis $\{e_n\}_{n\in\N}$ is given by
the Gramian
\begin{equation*}
A=\left(\begin{array}{lllll}1&2&0&\dots&\dots\\2&1&0&\dots&\dots\\0&0&1&0&\dots\\\vdots&\vdots&0&\ddots&\ddots\end{array}\right)
\end{equation*}
The central matrices  of $A$ are either the identity in $\ell^2$ or $A$ itself which are both invertible. Thus $\{v_n\}_{n\in\N}$ is woven with the canonical basis $\{e_n\}_{n\in\N}$. 

On the other hand when we write $A=D+R$ with $D$ diagonal we obtain that $D=Id$ and $Re_2=2e_1$ so that $2\norm{R}>\sup_n\abs{d_{n,n}}$.

As in the finite case, there is a description of the class $\Ww$ in terms of a reconstruction result.

\begin{proposition}\label{infinito}
Assume that $A\in \Ww$ and let $J\subset I$. Then any $x\in\ell^2(I)$ can be recovered from the samples $Y_J(x)=\{(Ax)(j)\}_{j\in J}\cup\{x(j)\}_{j\in I\setminus J}$. Conversely, if $A$ is a matrix in $I$ and the maps $x\mapsto Y_J(x)$ are injective and uniformly bounded for all $J\subset I$ then $A\in\Ww$.
\end{proposition}
\begin{proof}
If $A$ is a bounded matrix in $\Ww$ and $J\subset I$, then the matrix $A(J)$ defined for $j \in I$, whose $j^{\rm th}$-row is equal to $Ae_j$ when $j\in J$ and $e_j$ when $j\in I\setminus J$ defines a bounded operator in $\ell^2(I)$ with $A(J)x=Y_J(x)$. 

Further, when $A \in \Ww$, its columns $\{a_i\}_{i\in I}$ form a woven Riesz basis with the canonical orthonormal basis $\{e_i\}_{i\in I}$, as $A$ is an invertible matrix that represents the change of bases matrix between  $\{e_i\}_{i\in I}$ and $\{a_i\}_{i\in I}$. Thus, in such case, the columns of $A(J)$ form a Riesz basis and in particular $A(J)$ is invertible so that $x=A(J)^{-1}(A(J)(x))$ for all $x\in\ell^2(I)$.

Conversely, if $x\mapsto Y_J(x)$ defines a bounded injective operator in $\ell^2(I)$ then multiplication by $A(J)$ is also a bounded injective operator in $\ell^2(I)$. If the lower and upper bounds of $Y_J$ are independent of $J$, then the bounds of $A(J)$ are independent of $J$ and therefore its columns form a Riesz sequence with uniform bounds. But from \cite[Theorem 5.2]{weaving} this is enough to claim that $\{e_i\}_{i\in I}$ and $\{a_i\}_{i\in I}$ are woven and thus its change of basis $A$ is in the class $\Ww$.
\end{proof}

Since any Hilbert space can be isomorphically linked to an $\ell^2(I)$ space through any of its Riesz bases, we can also adapt definition of woven with respect to a given basis. 

\begin{definition}
Let $T:\Hh\to\Hh$ a bounded operator on a Hilbert space $\Hh$ and let $v=\{v_i\}_{i\in I}$ be a Riesz basis of $H$. We say that $T$ is {\it weaving } with respect to $v$ if $T_v=(t_{i,j})_{i,j\in I}\in\Ww$, where
\begin{equation*}
T_vv_i=\sum_{j\in I}t_{j,i}v_j
\end{equation*}
holds for each $i\in I$. In other words $T_v$ is a matrix operator in $\ell^2(I)$ corresponding to $T$ and the basis $v$.
\end{definition}

Note that $T:\Hh\to\Hh$ is weaving with respect to a Riesz basis $v=\{v_i\}_{i\in I}$ if and only if $Tv=\{Tv_i\}_{i\in I}$ is a Riesz basis woven with $v$. From Proposition \ref{infinito} we have the following result.

\begin{corollary}
Let $T:\Hh\to\Hh$ a woven operator on a Hilbert space $\Hh$ with respect to a Riesz basis $v=\{v_i\}_{i\in I}$ of $H$.  Then for all $J\subset I$, every  $f\in\Hh$ can be recovered from $\{P_{J,v}f\}\cup\{P_{I\setminus J,v}Tf\}$, where $P_{J,v}$ denotes the projection onto $\overline{\text{span}}\{v_i:i\in J\}$. Conversely, if $T:\Hh\to\Hh$ is a bounded operator and $v=\{v_i\}_{i\in I}$ Riesz basis such that for all $J\subset I$, every  $f\in\Hh$ can be recovered from $\{P_{J,v}f\}\cup\{P_{I\setminus J,v}Tf\}$, then $T$ is weaving with respect to $v$.
\end{corollary}
\begin{proof}
The proof is straightforward using that $T=D_v^{-1}T_vD_v$ where, as before $D_v:\Hh\to\ell^2(I)$ is the coefficient map. i. e. $D_v(f)=\{c_s\}_{s\in I}$ for $f=\sum_{s\in I}c_sv_s$.
\end{proof}

To end this section we list some of the properties whose proof can be carried over from Proposition \ref{propertiesf} to the infinite dimensional case. We omit the proofs here as they are very similar.

\begin{proposition}
Let $A:\ell^2(I)\to\ell^2(I)$ be a bounded operator
\begin{enumerate}
\item $A\in\Ww$ if and only if $A$ is invertible and $A^{-1}\in\Ww$.

\item $A\in\Ww\Leftrightarrow A^t \in\Ww\Leftrightarrow A^*\in\Ww$.

\item If $A\in\Ww$ and $B$ is either a diagonal invertible or a permutation operator then $B^*AB\in\Ww$.

\item If $A\in\Ww$ and $B$ is a central sub-matrix of $A$, then $B$ is in $\Ww$.
\end{enumerate}
\end{proposition}

From $(4)$ we obtain that if $\{v_i\}_{i\in I}$ is a Riesz basis woven with the canonical basis $\{e_i\}_{i\in I}$, then for all $J\subset I$, $\{(E_J^*)v_i\}_{i\in J}$ and $\{e_i\}_{i\in J}$ are woven in $\ell^2(J)$ where, as in the proof of Theorem \ref{teo20}, $E_J^*:\ell^2(I)\to\ell^2(J)$ is the restriction. 
In fact, in such case the change of basis matrix from $\{E_J^*v_i\}_{i\in J}$ to $\{e_i\}_{i\in J}$ is a sub-matrix of the one from $\{v_i\}_{i\in I}$ to $\{e_i\}_{i\in I}$. 

Actually more can be said. In general if two frames $\{v_i\}_{i\in I}$ and $\{w_i\}_{i\in I}$ are woven in a Hilbert space $\Hh$ and $T:\Hh\to\Hh$ is a bounded operator with closed range then $\{Tv_i\}_{i\in I}$ and $\{Tw_i\}_{i\in I}$ are woven in $T(\Hh)$.

\section{Weaving in SIS}\label{SIS}
We now show a result about woven Riesz basis in shift-invariant spaces.

\begin{definition}
We say that a closed subspace $V\subset L^2(\R^d)$ is shift-invariant if $T_kf(\cdot):=f(\cdot-k)\in V$ for all $f\in V$ and $k\in\Z^d$.

A subset of functions $\Phi\subset V$ is said to be a Riesz generator set of $V$ if $\{T_k\phi: \phi\in\Phi, k\in\Z^d\}$ is a Riesz basis of $V$.
\end{definition}

\begin{proposition}
The map $\tau:L^2(\R^d)\to L^2(\T^d ,\ell^2(\Z^d))$ defined by 
\begin{equation*}
\tau f(\omega):=\{\widehat{f}(\omega+k)\}_{k\in\Z^d}
\end{equation*}
is an isometric isomorphism.
\end{proposition}

We will denote by $\widehat{V}=\tau V$, the image of a shift-invariant space $V$ under $\tau$.

\begin{definition}
A range function is a mapping
\begin{equation*}
J:\T^d\to\{\text{ closed subspaces of }\ell^2(\Z^d)\}
\end{equation*}
that is measurable in the sense that its projections are weakly measurable functions.

Given a shift invariant space $V$, its associated range function is $J_V(\omega)=\{\tau f(\omega): f\in V\}$. When there is no confusion we will just denote this function by $J$.
\end{definition}
For a comprehensive  study about characterizations of SIS using range functions see \cite{bow} and references therein.

\subsection{Weaving Riesz generators}\label{Riesz generators}
In this section we give conditions in order that the weavings of two  Riesz generator sets of a SIS preserve the property of being  a Riesz generator set.
\begin{definition}
Let $V\subset L^2(\R^d)$ be a shift invariant space and suppose it admits a Riesz generator set. If $\{\phi_1, \dots, \phi_n\}$ and $\{\psi_1, \dots, \psi_n\}$ are two Riesz generator sets, we say that they are woven if for all $I\subset\{1, \dots, n\}$ the set $\{\phi_i\}_{i\in I}\cup\{\psi_i\}_{i\notin I}$ is a Riesz generator set.
\end{definition}

\begin{proposition}
In the same context as above, $\{\phi_1 , \dots,  \phi_n\}$ and $\{\psi_1 , \dots,  \psi_n\}$ are woven if and only if there exists $A,B>0$ such that $\{\tau\phi_1(\omega), \dots, \tau\phi_n(\omega)\}$ and $\{\tau\psi_1(\omega), \dots,  \tau\psi_n (\omega)\}$ are woven in $J(\omega)$ with frame bounds $A$ and $B$ for almost every $\omega$.
\end{proposition}
\begin{proof}
This follows from the definition and \cite[Theorem 2.3]{bow}.
\end{proof}

Now, due to Theorem \ref{elteo}, we know that $\{\tau\phi_1(\omega), \dots, \tau\phi_n(\omega)\}$ and $\{\tau\psi_1(\omega), \dots, \tau\psi_n(\omega)\}$ are woven in $J(\omega)$ if and only if its change of basis matrix $A_\omega\in\Ww_n$. There is a natural way to lift this matrix to a linear map in $J(\omega)$: define the change of basis map as
\begin{equation*}
R_\omega:J(\omega)\to J(\omega)\text{ with }R_\omega(\tau\phi_i(\omega))=\tau\psi_i(\omega), i=1, \dots, n,
\end{equation*}
and coordinate map for $\Phi$ as
\begin{equation*}
T_\omega^\Phi:J(\omega)\to\C^n\text{ with }T_\omega^\Phi(\tau\phi_i(\omega))=e_i, i=1, \dots, n.
\end{equation*}
Then the following diagram commutes
\begin{equation*}
\begin{tikzcd}
J(\omega)\arrow{r}{R_\omega}\arrow[swap]{d}{T_\omega^\Phi}&J(\omega)\arrow{d}{T_\omega^\Phi}
\\
\C^n\arrow{r}{A_\omega}&\C^n
\end{tikzcd}
\end{equation*}
Further lifting to $\widehat{V}$ we can define $\widetilde{R}:\widehat{V}\to\widehat{V}$ by
\begin{equation*}
\widetilde{R}\widetilde{f}(\omega):=R_\omega(\widetilde{f}(\omega))
\end{equation*}
for each $\widetilde{f}\in\widetilde{V}$ and $\omega\in\T^d$. Finally, going back using $\tau$, this defines an operator $R:V\to V$ with $\tau(Rf)=\widetilde{R}\widetilde{f}$. Thus $R$ commutes with all translates $T_k$, $k\in\Z^d$. Further, note that it is the only one with this property such that $R\phi_i=\psi_i$ for all $i=1, \dots, n$.

\subsection{Weaving Riesz bases of translations}

In the previos section,
we gave necessary and sufficient conditions in order that two Riesz generator sets  of a SIS are woven.
The weavings in that case interchange the generators with all their translations. In this section we will look at the general case where we allow to interchange any translate
of a generator with the translate of a possibly different generator. That is, we study when two Riesz bases of a SIS, given by the translations of generator sets in a SIS, $\{t_k\phi_j\}_{k,j}$ and $\{t_k\psi_j\}_{k,j}$ are woven as Riesz basis.

Riesz generator sets can be characterized through their Grammians as shown in \cite{structure}. Suppose $\Phi=\{\phi_1, \dots, \phi_n\}$ is a finite set of functions in $L^2(\R^d)$ and let $G_\Phi:[-1/2,1/2)^d\to\C^{n\times n}$ be given by
\begin{equation*}
G_\Phi(\zeta)_{j,l}=\sum_{m\in\Z^d}\widehat{\phi_j}(\zeta+m)\overline{\widehat{\phi_l}(\zeta+m)}
\end{equation*}
for each $j,l=1, \dots, n$. Thus, $G_\Phi(\zeta)$ is a positive semidefinite matrix since
\begin{equation*}
\scal{v,G_\Phi(\zeta)v}=\sum_{m\in\Z^d}\abs{\sum_{j=1}^nv_j\widehat{\phi_j}(\zeta+m)}^2
\end{equation*}
holds for all $v=(v_1, \dots, v_n)\in\C^n$ and almost all $\zeta\in[-1/2,1/2)^d$. Further, it can be proved \cite{structure} that $\Phi$ is a Riesz generator of $\overline{\text{span}}\{t_k\phi_j:k\in\Z^d, j=1, \dots, n\}$ if and only if $G_\Phi(\zeta)$ is positive definite with uniform bounds, i.e. there exists $A, B > 0$ such that
\begin{equation*}
A\norm{v}^2\leqslant\scal{v,G_\Phi(\zeta)v}\leqslant B\norm{v}^2
\end{equation*}
holds for all $v\in\C^n$ and almost all $\zeta\in[-1/2,1/2)^d$. When $\Phi$ has only one element $\phi$ this last condition reads as
\begin{equation*}
A\leqslant\sum_{m\in\Z^d}\abs{\widehat{\phi}(\zeta)}^2\leqslant B\quad\text{a.e. }\zeta\in[-1/2,1/2)^d.
\end{equation*}

We are now ready to prove our wieving result in this context. For simplicity, we only prove the case of generator sets with one element and in $L^2(\R)$ as the proof of the more general case is similar.

\begin{proposition}\label{pwn}
Let $V\subset L^2(\R)$ be a shift-invariant space with a single Riesz generator $\phi\in L^2(\R)$ with lower and upper Riesz bounds $A$ and $B$ respectively. 
If $\psi\in V$ is such that there exists a constant $\mu>0$ for which
\begin{equation*}
\sum_{k\in\Z}\abs{\widehat{\phi}(\zeta+k)-\widehat{\psi}(\zeta+k)}^2\leqslant\mu<A<\sum_{k\in\Z}\abs{\widehat{\phi}(\zeta+k)}^2\quad\text{a.e. }\zeta\in[-1/2,1/2),
\end{equation*}
then $\{t_k\psi:k\in\Z\}$ is a Riesz basis woven with $\{t_k\phi:k\in\Z\}$.
\end{proposition}

In order to prove this proposition we will use the following perturbation result.

\begin{theorem}[\cite{Christ}]
Let $\{x_k\}_{k\in\N}$ be a frame of a Hilbert space $\Hh$ with lower constant $A$ and let $\{y_k\}_{k\in\N}$ be a sequence in $\Hh$. Suppose there exists $\mu<A$ such that
\begin{equation}\label{Christensen}
\norm{\sum_{k\in N}c_k(x_k-y_k)}^2_\Hh\leqslant\mu\norm{c}^2_{\ell^2}
\end{equation}
holds for all finite sequences $c=(c_k)_{k\in\N}$. Then $\{y_k\}$ is a frame for $\Hh$.
\end{theorem}

\begin{proof}[Proof of Proposition \ref{pwn}]
To compare $\{t_k\psi:k\in\Z\}$ with the Riesz basis $\{t_k\phi:k\in\Z\}$ we take $c=(c_1,\dots,c_n)$ and compute
\begin{align*}
\norm{\sum_{k=1}^nc_k (t_k\phi-t_k\psi)}^2&=\bigintss_\R\abs{\sum_{k=1}^nc_k(\phi(\zeta-k)-\psi(\zeta-k)}^2\d\zeta
\\
&=\bigintss_\R\abs{\sum_{k=1}^nc_ke^{-2\pi i k\zeta}}^2\abs{\widehat{\phi}(\zeta)-\widehat{\psi}(\zeta)}^2\d\zeta.
\end{align*}

by using Plancharel's Theorem. Next, decomposing the integral into equal intervals we get
\begin{align*}
\bigintss_\R\abs{\sum_kc_ke^{-2\pi i k\zeta}}^2&\abs{\widehat{\phi}(\zeta)-\widehat{\psi}(\zeta)}^2\d\zeta=\sum_{n\in\Z}\bigintss_{n-1/2}^{n+1/2}\abs{\sum_kc_ke^{-2\pi i k\zeta}}^2\abs{\widehat{\phi}(\zeta)-\widehat{\psi}(\zeta)}^2\d\zeta
\\
&=\sum_{m\in\Z}\bigintss_{-1/2}^{1/2}\abs{\sum_kc_ke^{-2\pi i k\zeta}}^2\abs{\widehat{\phi}(\zeta+m)-\widehat{\psi}(\zeta+m)}^2\d\zeta
\\
&=\bigintss_{-1/2}^{1/2}\abs{\sum_kc_ke^{-2\pi i k\zeta}}^2\sum_{m\in\Z}\abs{\widehat{\phi}(\zeta+m)-\widehat{\psi}(\zeta+m)}^2\d\zeta
\\
&\leqslant\mu\bigintss_{-1/2}^{1/2}\abs{\sum_kc_ke^{-2\pi i k\zeta}}^2\d\zeta=\mu\norm{c}^2_{\ell^2}
\end{align*}
where in the last line, the inequality  follows by hypothesis and the equality from the fact that $\{e^{-2\pi i k\zeta}\}_k$ is an orthonormal basis of $L^2(-1/2,1/2)$. 

Applying (\ref{Christensen}) and its conclusion,  we obtain that $\{t_k\psi:k\in\Z\}$ is a frame of $V$. 

Further, the same computation shows that for any chosen set $I\subset\Z$, $\{t_k\psi:k\in I\}\cup\{t_k\phi:k\in\Z\setminus I\}$ is also frame of $V$. 

Hence, $\{t_k\phi:k\in\Z\}$ and $\{t_k\psi:k\in\Z\}$ are woven frames in $V$. But since $\{t_k\phi:k\in\Z\}$ is actually a Riesz basis, \cite[Theorem 5.4]{weaving} tells us that $\{t_k\psi:k\in\Z\}$ is actually a Riesz basis -and so are all the weaving choices between $\{t_k\psi:k\in\Z\}$ and $\{t_k\phi:k\in\Z\}$, cf. \cite[Theorem 5.3]{weaving}.
\end{proof}

We now apply the result to the case of a Paley-Wiener space.
\begin{corollary}
Let
\begin{equation*}
PW_{1/2}=\{f\in L^2(\R):\widehat{f}(\zeta)=0\text{ if }\zeta\notin[-1/2,1/2)\}
\end{equation*}
be the classic Paley-Wiener space. If $\psi\in PW_{1/2}$ is such that $\abs{1-\widehat{\psi}(\zeta)}<1$ a.e. $\zeta\in[-1/2,1/2)$, then $\{t_k\psi:k\in\Z\}$ is a Riesz basis of $PW_{1/2}$ which is woven with the orthonormal basis $\{t_k\text{sinc}:k\in\Z\}$ ($\text{sinc}(x)=\frac{\sin\pi x}{\pi x}$).

In particular if $\{t_k\phi:k\in\Z\}$ is a Riesz basis of $PW_{1/2}$ with $A<\abs{\widehat{\phi}(\zeta)}^2<B$ a.e. $\zeta\in[-1/2,1/2)$ and for some constants $A,B>0$ with $\max(B-1,1-A)<1$, then $\{t_k\phi:k\in\Z\}$ is woven with $\{t_k\text{sinc}:k\in\Z\}$
\end{corollary}
\begin{proof}
This follows from the previous result and the fact that $\widehat{\text{sinc}}(\zeta)=\chi_{[-1/2,1/2]}(\zeta)$.
\end{proof}

As said before, analogous computations give the result for a SIS with a finite Riesz generator set. We omit the proof.

\begin{proposition}
Let $V\subset L^2(\R^d)$ a shift-invariant space with a finite Riesz generator set $\Phi=\{\phi_1, \dots, \phi_n\}\subset L^2(\R^d)$, with lower and upper Riesz bounds $A$ and $B$ respectively. Let $\Psi=\{\psi_1, \dots, \psi_n\}\subset V$ be another set of functions such that there exists a constant $\mu>0$ for which
\begin{equation*}
c^*G_{\Phi-\Psi}(\zeta)c\leqslant\mu<A<c^*G_\Phi(\zeta)c
\end{equation*}
holds for all $c\in\C^n\setminus\{0\}$ and a.e. $z\in[-1/2,1/2)^d$, where $\Phi-\Psi=\{\phi_1-\psi_1, \dots, \phi_n-\psi_n\}$. Then $\{t_k\psi_j:j\in\{1, \dots, n\},k\in\Z\}$ is a Riesz basis that is woven with $\{t_k\phi_j:j\in\{1, \dots, n\},k\in\Z^d\}$.
\end{proposition}

\end{document}